\documentclass[12pt,a4paper]{amsart}
\usepackage{graphics}
\usepackage{epsfig}
\usepackage{graphicx}
\usepackage[all]{xy}
\theoremstyle{plain}
\usepackage{amssymb}

\advance\hoffset-20mm \advance\textwidth41mm

\newcommand{\p}{\partial}

\newtheorem{theorem}{Theorem}
\newtheorem{lemma}{Lemma}
\newtheorem{proposition}{Proposition}
\newtheorem{corollary}{Corollary}

\theoremstyle{definition}
\newtheorem{definition}{Definition}
\newtheorem*{definition*}{Definition}
\newtheorem{example}{Example}
\newtheorem{remark}{Remark}

\def\AA{{\mathbb A}}

\def\QQ{{\mathbb Q}}

\def\ZZ{{\mathbb Z}}

\def\GG{{\mathbb G}}

\def\TT{{\mathbb T}}
\def\kk{{\mathbf k}}

\def\<{{\langle}}
\def\>{{\rangle}}

\def\ad{\mathop{\rm ad}}
\def\LND{\mathop{\rm LND}}
\def\Ker{\mathop{\rm Ker}}
\def\Spec{\mathop{\rm Spec}}
\def\Der{\mathop{\rm Der}}
\def\Nil{\mathop{\rm Nil}}
\def\Frac{\mathop{\rm Frac}}
\def\deg{\mathop{\rm deg}}
\def\Aut{\mathop{\rm Aut}}
\def\exp{\mathop{\rm exp}}
\def\diag{\mathop{\rm diag}}
\def\Lie{\mathop{\rm Lie}}

\begin{document}
\sloppy
\title{On sums of homogeneous locally nilpotent derivations}
\author{Elena Romaskevich}
\address{Department of Higher Algebra, Faculty of Mechanics and Mathematics,
Moscow State University, Leninskie Gory 1, GSP-1, Moscow, 119991,
Russia}
\email{lena.apq@gmail.com}

%
\begin{abstract}
Let $A$ be a commutative associative integrally closed $\kk$-algebra
without zero divisors effectively graded by a lattice. We obtain a criterion of local nilpotency of the sum of two
homogeneous locally nilpotent derivations (LNDs) of fiber
type on $A$ in terms of their degrees. The same problem is solved for commutators of two homogeneous LNDs.
\end{abstract}
\maketitle

\section*{Introduction}
Let $\kk$ be an algebraically closed field of characteristic zero. We
consider an algebraic torus $\TT\simeq(\kk^{\times})^n$ acting effectively
on a normal affine variety $X$ and the corresponding grading of the algebra
$A=\kk[X]$ by the lattice $M$ of characters of $\TT$.

In this paper we study some properties of homogeneous locally
nilpotent derivations (LNDs). A~derivation $\p$ on $A$ is
said to be {\it locally nilpotent} if for each $a\in A$ there exists
$n\in \ZZ_{>0}$ such that $\p^n(a)=0$. LNDs on $A$ are in one-to-one
correspondence with regular actions of the group $\GG_a(\kk)=(\kk,+)$ on $X$, see \cite{GF}. It is easy to see that a derivation $\p$ on $A$ is
homogeneous if and only if the corresponding $\GG_a(\kk)$-action is
normalized by the torus $\TT$.

We use a description of homogeneous LNDs on an $M$-graded
algebra~$A$. Recall that a homogeneous LND on $\kk[X]$ is said to be
of {\it fiber type} if $\p(\kk(X)^{\TT})=0$, see Definition
$\ref{def-fiber/horizontal}$. In geometric terms, $\p$ is of fiber
type if and only if generic orbits of the corresponding
$\GG_a(\kk)$-action are contained in the closures of $\TT$-orbits. A complete classification of
homogeneous LNDs of fiber type is due to A.~Liendo, see \cite{AL-2}.

A problem that one faces when dealing with locally nilpotent
derivations is that the set of all LNDs on an algebra $A$ admits no
obvious algebraic structure. In Epilogue to~\cite{GF} G.~Freudenburg
poses several natural questions concerning the structure of
$\LND(A)$. Namely, given $\p_1,\p_2\in \LND(A)$, under what conditions
are $[\p_1,\p_2]$ and $\p_1+\p_2$ locally nilpotent? These questions
are still open in general. Some results were otained by M.~Ferrero, Y.~Lequain, and
A.~Nowicki in \cite{MF+YL+AN}. Namely, given two commuting locally nilpotent
derivations $d, \delta$ $\in \Der(R)$ and an element
$a \in R$, it is proven that the derivation $ad + \delta$ is locally
nilpotent if and only if $d(a) = 0$.

In this note we give a complete answer to the above-mentioned questions for homogeneous locally
nilpotent derivations of fiber type. Namely, $\p_1+\p_2$ (resp. $[\p_1,\p_2]$) is locally nilpotent if and only if the sum $\deg\,\p_1+\deg\,\p_2$ of their degrees is not in the weight cone $\omega_M(A)$, see Theorem $\ref{main_theorem}$ (resp. Proposition $\ref{commutator}$).

It should be noted that study of algebraic properties of homogeneous locally nilpotent derivations plays an important role in recent works on automorphisms of algebraic varieties, see, e.g. \cite{IA+AL}, \cite{IA+HF+SK+FK+MZ}, \cite{IA+JH+EH+AL}. One more motivation comes from the question posed by V.~Popov, see \cite[Problem 3.1]{VP}. He considers two locally nilpotent derivations $\p_1$ and $\p_2\in \LND(\kk[x_1,\dots,x_n])$ and asks when the minimal closed subgroup of $\Aut_{\kk}\kk[x_1,\dots,x_n]$ containing the one dimensional subgroups $\{\exp(t\p_1)\ |\ t\in\kk\}$ and $\{\exp(t\p_2)\ |\ t\in\kk\}$ is finite dimensional. Here it is important to know when $\p_1+\p_2$ and $[\p_1,\p_2]$ are locally finite.

In Section $\ref{sec1}$ we collect some basic definitions and facts
about LNDs. In Section $\ref{sec2}$ we recall generalities
on $\TT$-varieties and corresponding gradings on their coordinate
algebras. Section $\ref{sec3}$ is devoted to background on
homogeneous LNDs and to the classification of LNDs of fiber
type from \cite{AL-2}. In Sections $\ref{sec4}$ and
$\ref{sec5}$ we study commutators and sums of homogeneous LNDs of
fiber type. Section $\ref{sec6}$ contains some corollaries of Theorem $\ref{main_theorem}$, examples and generalizations.

\section{Locally nilpotent derivations}
\label{sec1}

Let $A$ be a commutative associative $\kk$-algebra without zero
divisors. A {\it derivation} $\p: A\to A$ is a linear map satisfying the
Leibniz rule:
$$
\p(ab)=a\p b+b\p a \quad \mathrm{for\ all}\ a,\,b \in A
$$

Denote the set of all derivations on $A$ by $\Der(A)$.

Recall that an algebra is said to be {\it graded} by a commutative
semigroup $S$ if there is a direct sum decomposition $A=\bigoplus_{s\in S}A_s$ such that $A_s\cdot A_{s'}\subseteq
A_{s+s'}$ for all $s,\ s'\in S$. A~derivation $\p$ is called {\it
homogeneous} if it sends homogeneous elements to homogeneous ones.
We will write $\p(s)=s'$, for $s,\ s'\in S$, if
$A_s\nsubseteq\Ker(\p)$ and $\p(A_s)\subseteq A_{s'}$. By Leibniz
rule, for $a\in A_s \backslash\Ker(\p)$ and $b\in A_{s'}\backslash
\Ker(\p)$ we have
$$
\p(ab)=a\p b+b\p a \in A_{\p(s+s')},
$$
and so
$$
\p(s+s')=s+\p(s')=s'+\p(s).
$$
Thus, for a homogeneous nonzero derivation $\p$ there exists $s_0\in
S$ such that $\p A_s \subset A_{s+s_0}$ for all $s \in S$. An element
$s_0\in S$ is called the {\it degree} of $\p$ and is denoted by $\deg\,\p$.

\begin{definition}
A derivation $\p \in \Der(A)$ is called {\it locally nilpotent} (LND
for short) if for every $f\in A$ there exists $n\in \ZZ_{>0}$ such
that $\p ^nf=0$. The set of all locally nilpotent derivations on~$A$
is denoted by $\LND(A)$.
\end{definition}

We associate to a derivation $\p$ a set $\Nil(\p)=\{f\in A\ |\
\exists\,n\in\ZZ_{>0}: \p^n f=0\}$. Thus, $\p\in \LND(A)$ exactly when
$\Nil(\p)=A$.

Note that if $\p_1,\ \p_2\in \LND(A)$ and $[\p_1,\p_2]=0$, then $\p_1+\p_2\in\LND(A)$, i.e. the sum of commuting LNDs is an LND as well.

%

Recall that locally nilpotent derivations on an affine algebra $A$ are in one-to-one correspondence with regular actions of $\GG_a(\kk)$ on $X=\Spec\,A$. Indeed, we have a rational representation $\eta:\,\GG_a(\kk)\hookrightarrow \Aut_{\kk}(A)$, where $\eta(t)=\exp(tD)$. In geometric terms this means that $D$ induces a regular $\GG_a(\kk)$-action on $X$.
Conversely, let $\rho:\,\GG_a(\kk)\times X\to X$ be a regular $\GG_a(\kk)$-action. Then $\rho$ induces a locally nilpotent derivation $D=\frac{d}{dt}|_{t=0}\rho^*$, where $\rho^*:\, A\to A[t]$. For more detail see \cite[Section 1.5]{GF}.

\section{$\TT$-varieties}
\label{sec2}

An {\it algebraic torus} $\TT=\TT^n$ of dimension $n$ is the
algebraic variety $(\kk^{\times})^n$ with the natural structure of
algebraic group. A {\it $\TT$-variety} is an algebraic variety
endowed with an effective $\TT$-action.

A {\it character} (resp. {\it one-parameter subgroup}) of $\TT$ is a
homomorphism of algebraic groups $\chi:\,\TT\to \kk^{\times}$ (resp.
$\lambda:\,~ \kk^{\times}\to\TT$). The set of all characters (resp.
one-parameter subgroups) form a lattice $M$ (resp. $N$) of rank $n$.
For every $m\in M$ we denote by $\chi^m$ the corresponding character of $\TT$. We also let
$M_{\QQ}$ and $N_{\QQ}$ be the rational vector spaces $M\otimes_{\ZZ} \QQ$
and $N\otimes_{\ZZ} \QQ$.

Let $A$ be an {\it affine} algebra, i.e. a commutative associative
finitely generated $\kk$-algebra with unit and without zero divisors. It is well
known that effective $\TT$-actions on an affine variety $X=\Spec\,A$ are in
one-to-one correspondence with effective $M$-gradings on $A$. Thus, for a
$\TT$-variety $X$ we have an effective $M$-grading on
$A=\kk[X]$:
$$
A=\bigoplus_{m\in M}\widetilde{A_m}.
$$
Let $K=\Frac\,A$ be the field of fractions. We consider
$$
K_0=\left\{\frac{f}{g}\in K\ |\ \deg f=\deg g\right\}.
$$
Notice that $K_0$ coincides with the field of $\TT$-invariant
functions $\kk(X)^{\TT}$. Thus, we have a tower of field extensions
$\kk\subseteq \kk(X)^{\TT} \subseteq K$. One may represent
$\widetilde{A_m}=A_m\chi^m$, where $A_m\subseteq \kk(X)^{\TT}$. The
{\it weight cone} $\omega \subseteq M_{\QQ}$ of a given $M$-grading is a
cone in $M_{\QQ}$ spanned by the set $\{m\in M\ |\ A_m\ne 0\}$.
For a cone $\omega \subseteq M_{\QQ}$ we denote the set
$\omega\cap M$ by $\omega_M$. Finally, we have
$$
A=\bigoplus_{m\in \omega_M} A_m \chi^m, \quad A_m\subseteq
\kk(X)^{\TT}.
$$
Since $A$ is finitely generated, the cone $\omega$ is
polyhedral, and since the $M$-grading is effective, $\omega$ is
of full dimension.

{\it Complexity} of a $\TT$-action is the transcendence degree
of the field of $\TT$-invariant rational functions $\kk(X)^{\TT}$ over $\kk$.
In geometric terms, complexity of a $\TT$-action equals the
codimension of the generic $\TT$-orbit. In particular, for a
$\TT$-variety of complexity zero, $\kk(X)^{\TT}=\kk$ and $A\subseteq
\kk[M]$, where $\kk[M]$ stays for the group algebra of the lattice $M$ and is isomorphic to the algebra of Laurent polynomials over $\kk$. A {\it toric variety} is a normal $\TT$-variety of
complexity zero, or, equivalently, $\TT$ acts with an open orbit. From now on we will consider only normal
$\TT$-varieties.

\section{Demazure roots and homogeneous LNDs}
\label{sec3}

Let $M$ and $N$ be the lattices of characters and one-parameter subgroups of a torus $\TT$. We consider the natural pairing $\<\cdot,\cdot\>: M\times N\to \ZZ$ given by
$$
\<\chi,\lambda\>=l, \quad {\rm if} \quad \chi\circ\lambda(t)=t^l.
$$
This pairing extends in an obvious way to a pairing
$$
\<\cdot,\cdot\>: M_{\QQ}\times N_{\QQ}\to \QQ
$$
between $\QQ$-vector spaces. Let $A$ be as before an $M$-graded affine
algebra with the weight cone $\omega\subseteq M_{\QQ}$. Let
$\omega^{\vee}=\sigma\subseteq N_{\QQ}$ be the dual cone. Since
$\omega$ is of full dimension, $\sigma$ is a pointed polyhedral cone.

\begin{lemma}\cite[Lemma 1.13]{AL-1}
For any homogeneous LND $\p$ on $A$ the following holds.
\begin{itemize}
\item[1)] The derivation $\p$ extends in a unique way to a homogeneous
$\kk$-derivation on $\kk(X)^{\TT}[M]$.
\item[2)] If $\p(\kk(X)^{\TT})=0$ then the extension of $\p$ as in 1)
restricts to a homogeneous locally nilpotent
$\kk(X)^{\TT}$-derivation on $\kk(X)^{\TT}[\omega_M]$.
\end{itemize}
\end{lemma}

\begin{definition}\label{def-fiber/horizontal}
A homogeneous LND $\p$ on $A$ is said to be of {\it fiber type} if
$\p(\kk(X)^{\TT})=0$ and of {\it horizontal type} otherwise.
Derivations of fiber type correspond to $\GG_a(\kk)$-actions such
that generic $\GG_a(\kk)$-orbits are contained in the closures of
$\TT$-orbits.
\end{definition}

In \cite{AL-2} A.Liendo gave a complete classification of
homogeneous LNDs of fiber type on an arbitrary graded integrally closed affine algebra.
We need the following notation to present his results.



For a ray $\rho$ of a cone $\sigma$ we let $\sigma_{\rho}$ denote the cone spanned by all the rays of $\sigma$ except~$\rho$. From now on we denote by $\rho$ both a ray and its primitive
vector. We also let
$$
S_{\rho}=\sigma_{\rho}^{\vee}\cap\{e\in M\ |\ \<e,\rho\>=-1\}.
$$

In other words, $S_{\rho}$ is the set of lattice vectors
$e\in M$ such that $\<e,\rho\>=-1$ and $\<e,\rho^{'}\>\geq 0$ for
every other ray $\rho^{'} \subseteq \sigma$.

\begin{definition}(see \cite{MD})
In the above notation the elements of the set
$\Re=\bigcup_{\rho\subseteq\sigma} S_{\rho}$ are called {\it
Demazure roots} of the cone $\sigma$.
\end{definition}

Let $e\in S_{\rho}$ be a Demazure root corresponding to a ray $\rho$
of the cone $\sigma$. Set
$$
\Phi_e=\{f\in K\ |\ f\cdot A_m\subseteq A_{m+e}\},\ \Phi_e^{\times}=\Phi_e\backslash\{0\}.
$$
We define a homogeneous derivation $\p_{\rho,e}$ of degree $e$
on~$\kk(X)^{\TT}[M]$ by
$$
\p_{\rho,e}(\chi^m)=\<m,\rho\>\chi^{m+e}.
$$
Clearly, for $f\in\Phi_e$ the product $f\p_{\rho,e}$ is a derivation
on $\kk(X)^{\TT}[\omega_M]$. Finally, denote
$\p_{\rho,e,f}=f\p_{\rho,e}|_A$, which is a homogeneous LND on $A$.
The classification theorem is as follows.

\begin{theorem}\label{HLND}\cite[Theorem 2.4]{AL-2}
Every nonzero homogeneous LND $\p$ of fiber type on $A$ is of the
form $\p=\p_{\rho,e,f}$ for some ray $\rho\subseteq\sigma$, some
Demazure root $e\in S_{\rho}$ and some function
$f\in\Phi_e^{\times}$.
\end{theorem}

In particular, for any homogeneous locally nilpotent derivation $\p$ of fiber type we have $\deg \p\notin\omega_M$.

\begin{remark}
For a toric variety $X$ the closure of a dense $\TT$-orbit coincides with $X$, hence all
derivations on $A$ are of fiber type. In this particular case
Theorem $\ref{HLND}$ states that all nonzero homogeneous LNDs on $A=\kk[X]$ are
of the form $\p=\lambda\p_{\rho,e}$, for some $\lambda\in \kk^{\times}$.
\end{remark}


\begin{example}\label{ex-1}
With $N=\ZZ^n$ we let $\sigma$ be the cone in $N_{\QQ}$ spanned by
the basic unit vectors $e_i$. The dual cone $\omega\subseteq M_{\QQ}$ is
spanned by the basic unit vectors in $M_{\QQ}$ as well. The
corresponding semigroup algebra is a polynomial algebra $A=\kk[x_1,\dots,x_n],\ \deg x_i=e_i$, and the affine toric variety is $X=\AA^n$. According to Theorem $\ref{HLND}$ all nonzero homogeneous
LNDs on $A$ are:
$$
\p=\lambda x_1^{i_1}\dots \widehat{x_k^{i_k}}\dots
x_n^{i_n}\frac{\p}{\p x_k}, \quad \lambda\in \kk^{\times},\ i_1,\dots,i_n\in
\ZZ_{\geqslant 0} \quad (k=1,\dots,n).
$$

Degree of the preceding derivation is $(i_1,\dots,-1,\dots,i_n)$,
which is a Demazure root of~$\sigma$ corresponding to the ray
$\rho_k$.
\end{example}

\section{Commutators of two homogeneous LNDs}
\label{sec4}

Suppose $\p_1,\p_2\in \LND(A)$ are nonzero homogeneous derivations of
fiber type. It follows from Theorem $\ref{HLND}$ that they are
given by
$$
\p_1(\chi^m)=f_1\<m,\rho_1\>\chi^{m+e_1},\qquad
\p_2(\chi^m)=f_2\<m,\rho_2\>\chi^{m+e_2},
$$
where $e_1,e_2\in\Re$ and $f_1\in \Phi_{e_1}^{\times}$, $f_2\in \Phi_{e_2}^{\times}$.
According to the definition, we have
\begin{multline*}
[\p_1,\p_2](\chi^m)=f_1f_2(\<m+e_2,\rho_1\>\<m,\rho_2\>-\<m+e_1,\rho_2\>\<m,\rho_1\>)\chi^{m+e_1+e_2}=\\
=f_1f_2(\<e_2,\rho_1\>\<m,\rho_2\>-\<e_1,\rho_2\>\<m,\rho_1\>)\chi^{m+e_1+e_2}.
\end{multline*}

\begin{lemma}\label{sameray}
If $\p_1,\p_2\in \LND(A)$ correspond to the same ray of the cone
$\sigma$, then $\p_1$ and~$\p_2$ commute.
\end{lemma}

\begin{proof}
Let $\rho_1=\rho_2=\rho$. Then
$\<e_2,\rho\>\<m,\rho\>-\<e_1,\rho\>\<m,\rho\>=(-1)\<m,\rho\>-(-1)\<m,\rho\>=0$,
and hence the commutator equals zero.
\end{proof}

\begin{proposition}\label{commutator}
Suppose $\p_1,\p_2\in \LND(A)$ correspond to different rays $\rho_1,\rho_2$ of the
cone $\sigma$. Then the following holds:
\begin{itemize}
\item[1)] $[\p_1,\p_2]=0 \Leftrightarrow \<e_1,\rho_2\>=0$ and
$\<e_2,\rho_1\>=0$,
\item[2)] $[\p_1,\p_2]\in \LND(A)\Leftrightarrow\ e_1+e_2\notin\omega_M$, i.e. $\<e_1,\rho_2\>=0$ or $\<e_2,\rho_1\>=0$.
\end{itemize}
\end{proposition}

\begin{proof}
1) Let $[\p_1,\p_2]=0$. Then for all $m\in\omega_M$ we have
$\<e_2,\rho_1\>\<m,\rho_2\>=\<e_1,\rho_2\>\<m,\rho_1\>$. One can
choose $m\in \omega_M$ such that $\<m,\rho_1\>=0$ and
$\<m,\rho_2\>>0$. In this case $\<e_2,\rho_1\>\<m,\rho_2\>=0$
implies that $\<e_2,\rho_1\>=0$. Similarly $\<e_1,\rho_2\>=0$. The
inverse implication holds automatically.

2) Let us prove sufficiency. Without loss of generality assume
$\<e_1,\rho_2\>=0,\ \<e_2,\rho_1\>>0$. Then $\<e_1+e_2,\rho_2\>=-1,\
\<e_1+e_2,\rho_1\>\geqslant 0$. Therefore $e_1+e_2$ is a Demazure
root corresponding to the ray $\rho_2$ of the cone $\sigma$ and
$[\p_1,\p_2]=f \p_{\rho_2,e_1+e_2}$, where
$f=f_1f_2\<e_2,\rho_1\>\in \Phi_{e_1+e_2}^{\times}$.

To prove necessity assume that the commutator is an LND. Since its
degree equals $e_1+e_2$, there exists a ray $\rho^*$ of the cone
$\sigma$ such that $\<e_1+e_2,\rho^*\>=-1$. As $e_1,e_2\in
\sigma_{\rho}^{\vee}$ for $\rho\ne\rho_1, \rho_2$, one obtains
$\rho^*=\rho_1$ or $\rho_2$. This proves the assertion.
\end{proof}

\begin{remark}\label{deg_notin_cone}
In general, if $\p$ is a homogeneous derivation of a graded algebra $A=\bigoplus_{m\in \omega_M}\widetilde{A_m}$, and  $\deg\p\notin\omega_M$, then $\p$ is locally nilpotent. Indeed, for any $m\in\omega_M$ one can find $k\in\ZZ_{\geqslant 1}$ such that $(m+k\deg\p)\notin\omega_M$ and, thus, $\p^k(\widetilde{A_m})=0$.
\end{remark}

\section{Sums of two homogeneous LNDs}
\label{sec5}

In this section we establish a necessary and sufficient condition
of local nilpotency of the sum of two homogeneous LNDs. Consider two nonzero homogeneous LNDs $\p_1, \p_2$ of fiber type:
$$
\p_1(\chi^m)=f_1\<m,\rho_1\>\chi^{m+e_1},\qquad
\p_2(\chi^m)=f_2\<m,\rho_2\>\chi^{m+e_2},
$$
where $e_1,e_2\in\Re$ and $f_1\in \Phi_{e_1}^{\times}$, $f_2\in \Phi_{e_2}^{\times}$. Note that their sum is homogeneous if and only if both $\p_1$ and $\p_2$ are of
the same degree, i.e. $e_1=e_2$.

\begin{theorem}\label{main_theorem}
In the above notation let $\p_1$ and $\p_2$ be homogeneous LNDs of
fiber type. Then $\p_1+\p_2\in \LND(A)$ if and only if
$e_1+e_2\notin\omega_M$.
\end{theorem}

\begin{proof}
Let us prove sufficiency. The condition $e_1+e_2\notin\omega_M$ means that either
both $\p_1$ and $\p_2$ correspond to the same ray, or
$\<e_2,\rho_1\>=0$ (up to permutation of indices). In the first
case according to Lemma~$\ref{sameray}$ our derivations
commute and hence their sum is an
LND.

Let us consider the second case. It suffices to show that
for any $m\in \omega_M$ there exists $n\in\mathbb{Z}_{\geqslant 0}$ such that $(\p_1+\p_2)^n(\chi^m)=0$. We proceed by induction on the
parameter $\<m,\rho_1\>$.

Let $\<m,\rho_1\>=0$. This means that $m\in\rho_1^\perp\cap\omega_M$ and
$(\p_1+\p_2)(\chi^m)=\p_2(\chi^m)$. In addition, the equality
$\<e_2,\rho_1\>=0$ implies that $m+e_2\in\rho_1^\perp\cap\omega_M$.
Continuing these arguments and using local nilpotency of $\p_2$,
we obtain the required condition.

Now consider an arbitrary point $m\in\omega_M$. Since $\p_2\in \LND(A)$, there exists
$n\in\mathbb{Z}_{\geqslant 0}$ such that $\p_2^n(\chi^m)=0$.
Therefore all the images of $m$ under powers of $\p_2$, namely
$m,m+e_2,\dots,m+(n-1)e_2$, lie in the hyperplane $H$ perpendicular
to $\rho_1$ and given by the equation $\<x,\rho_1\>=\<m,\rho_1\>$.
Now we apply $(\p_1+\p_2)^n$ to $\chi^m$. One can easily see that
applying $\p_1$ to an arbitrary element of the hyperplane $H$, the image will be in the hyperplane given by the equation $\<x,\rho_1\>=\<m,\rho_1\>-1$. Since each summand (except for
$\p_2^n$) in $(\p_1+\p_2)^n$ contains $\p_1$ and $\p_2^n(\chi^m)=0$
and using the inductive hypothesis, we obtain local nilpotency of the
sum $\p_1+\p_2$.

Now let us prove the inverse implication. If
$\mathrm{rk}\,M=1$, then $\omega_M=\ZZ$ or $\ZZ_{\geqslant 0}$ and the
dual cone $\sigma$ is $0$ or $\QQ_{\geqslant 0}$ respectively.
In the first case $\sigma$ contains no ray, thus there are no
derivations of fiber type on $\kk[X]$. In the second case $\sigma$
consists of one ray and the condition $e_1+e_2\notin \omega_M$
follows immediately.

The case $\mathrm{rk}\,M=2$ follows from results
of P.~Kotenkova, see \cite{PK}. We recall some notation and results
from this article. Consider a one-dimensional subtorus
$T\subset\TT$ given by the equation $e_1-e_2=0$. The torus $T$ also
acts on $X$ and every $\TT$-homogeneous LND on $\kk[X]$ is
$T$-homogeneous as well. Recall that a $\TT$-root of an $M$-graded algebra $A$ is the degree of some homogeneous LND on $A$. Any $\TT$-root of $A=\kk[X]$ can be restricted to some $T$-root and we denote by $\pi$ the restriction map. Denote by $\Gamma_T\subseteq N_{\QQ}$ the
hyperplane, corresponding to the subtorus $T$. Let
$\<\cdot,m_T\>=0$ be the equation of the hyperplane $\Gamma_T$,
where $m_T \in M$. Evidently, the locally nilpotent derivation
$\p_1+\p_2$ restricts to a homogeneous LND with respect to the torus
$T$. Our case corresponds to Case 3.3 in \cite[Proposition 6]{PK}. It follows that all $T$-homogeneous LNDs of degree $\pi(e_1)$ are of the following form (see \cite[Proposition~5]{PK}):

$$
\p(\chi^m)=\chi^{m+e_2}(\alpha\<\rho_1,m\>\chi^{m_T}+\beta\<\rho_2,m\>)(\alpha\chi^{m_T}+\beta\<\rho_2,m_T\>)^{\<\rho_1,e_2\>},
$$
for some $\alpha, \beta\in \kk$. Now one can easily see that for
$\p=\p_1+\p_2$ the exponent $\<\rho_1,e_2\>$ has to vanish, so $\<\rho_1,e_1+e_2\>=-1$ and $e_1+e_2\notin\omega_M$.

Let us consider the general case $\mathrm{rk}M=n$.

\begin{lemma}\label{extension to alg.closure}
Denote by $\bar{k}$ the algebraic closure of a field $k$. The
derivation $\p$ on a $k$-algebra $A$ extends in a unique way to a
derivation $\tilde{\p}$ on $A\otimes_k \bar{k}$. In addition, $\p$
is locally nilpotent if and only if $\tilde{\p}$ is.
\end{lemma}

\begin{proof}
To prove the first assertion of the lemma we let $x\in\bar{k}$ and assume that $p(t)$ is the
minimal polynomial of $x$. Applying $\tilde{\p}$ to the
equation $p(x)=0$, we obtain $\tilde{\p}(x)p'(x)=0$, hence
$\tilde{\p}(x)=0$. Further, an extension of $\p$ to $A\otimes_k \bar{k}$ by linearity
is unique. Since
$\bar{k}\subset\Ker(\tilde{\p})$, local nilpotency of
$\tilde{\p}$ and $\p$ are equivalent.
\end{proof}

According to the previous lemma we can replace $\kk(X)^{\TT}$ with its algebraic closure $\overline{\kk(X)^{\TT}}$.

It is required to show that $e_1+e_2\notin \omega_M$. If the roots
$e_1,e_2$ correspond to the same ray, the condition holds
automatically. If $e_1,e_2$ correspond to different rays $\rho_1$
and $\rho_2$, the required condition is equivalent to
the following one: at least one of expressions $\<e_1,\rho_2\>$ or
$\<e_2,\rho_1\>$ vanishes. We carry out the proof by contradiction.

We assume that the vectors $e_1$ and $e_2$ are not collinear. In this case for
any $m\in \omega_M$ we denote by $\gamma_m$ a two-dimensional plane
passing through $m$ and spanned by vectors $e_1,e_2$. It is obvious
that images of $m$ under powers of the derivation
$\p_1+\p_2$ are in $\omega_M\cap\gamma_m$. It follows from the formulae
defining $\p_1$ and $\p_2$ that
$\gamma_m\cap\rho_1^\bot\ne\varnothing$ and
$\gamma_m\cap\rho_2^\bot\ne\varnothing$. If
$\gamma_m\subset\rho_1^\bot$ then $\<e_2,\rho_1\>=0$
and similarly for another permutation of the indices. Otherwise
$\mathrm{dim}\,(\gamma_m\cap\rho_1^\bot)=\mathrm{dim}\,(\gamma_m\cap\rho_2^\bot)=1$.

Now we assume that $m=e_1+e_2$. Denote the lines in which the plane
$\gamma_m$ intersects $\rho_1^{\perp}$ and $\rho_2^{\perp}$ by $l_1$
and $l_2$ respectively. It follows from local nilpotency of
$\p_1,\ \p_2$ that there exist $k_1,k_2\in\ZZ_{\geqslant 0}$ such
that
$$
m_1=e_1+e_2+k_1e_1\in\rho_1^{\perp}\cap\gamma_m, \qquad
m_2=e_1+e_2+k_2e_2\in\rho_2^{\perp}\cap\gamma_m.
$$
Obviously, $0\in\rho_1^{\perp}\cap\rho_2^{\perp}\cap\gamma_m$. Thus,
$l_1$ and $l_2$ are precisely the lines passing through $0$ and
$m_1$ or~$m_2$ respectively. If $l_1$ and $l_2$ are different, we set
$\omega_m=\gamma_m\cap\omega$. The two-dimensional
cone~$\sigma_m$ dual to $\omega_m$ can be considered as embedded in $N_{\QQ}$ and
spanned by the rays $\rho_1,\ \rho_2$. One can easily see that $e_1$
and $e_2$ are Demazure roots of $\sigma_m$. A locally nilpotent
derivation on $\overline{\kk(X)^{\TT}}[\omega_M]$ can be restricted to an LND on
$\overline{\kk(X)^{\TT}}[\omega_m\cap M]$, hence using the assertion of the theorem
in case of dimension two we obtain $e_1+e_2\notin \omega_m\cap
M$, a contradiction.

If the lines $l_1$ and $l_2$ coincide, then the vectors $e_1+e_2+k_1e_1$ and
$e_1+e_2+k_2e_2$ are collinear. This means the following:
$$
\frac{k_1+1}{1}=\frac{1}{k_2+1} \Leftrightarrow (k_1+1)(k_2+1)=1
\Leftrightarrow k_1=k_2=0.
$$
Thus, $e_1+e_2\in \rho_1^{\perp}\cap\rho_2^{\perp}$ and $\gamma_m\cap\omega$ is the line spanned by $e_1+e_2$. Indeed, let
$\<ae_1+be_2,\rho_1\>\geqslant 0$. Using that $\<e_1,\rho_1\>=~-1,\
\<e_2,\rho_1\>=1$ one obtains $b\geqslant a$. Similarly we obtain
$a\geqslant b$ and, hence, $a=b$. Consider now an arbitrary element
$s\in\omega_M^0=\mathrm{Int}(\omega_M)$ and the corresponding
plane~$\gamma_s$. Since $\gamma_s
\parallel \gamma_m$, the plane $\gamma_s$ intersects both $\rho_1^{\perp}$ and
$\rho_2^{\perp}$ in parallel lines $v_1$ and $v_2$ with leading
vector $e_1+e_2$. Our aim now is to construct $f\in
\overline{\kk(X)^{\TT}}[\omega_M]$ such that $f\notin \Nil(\p_1+\p_2)$.


\begin{figure}[h]\label{fig-stripe}
\centerline{\includegraphics[scale = 0.25]{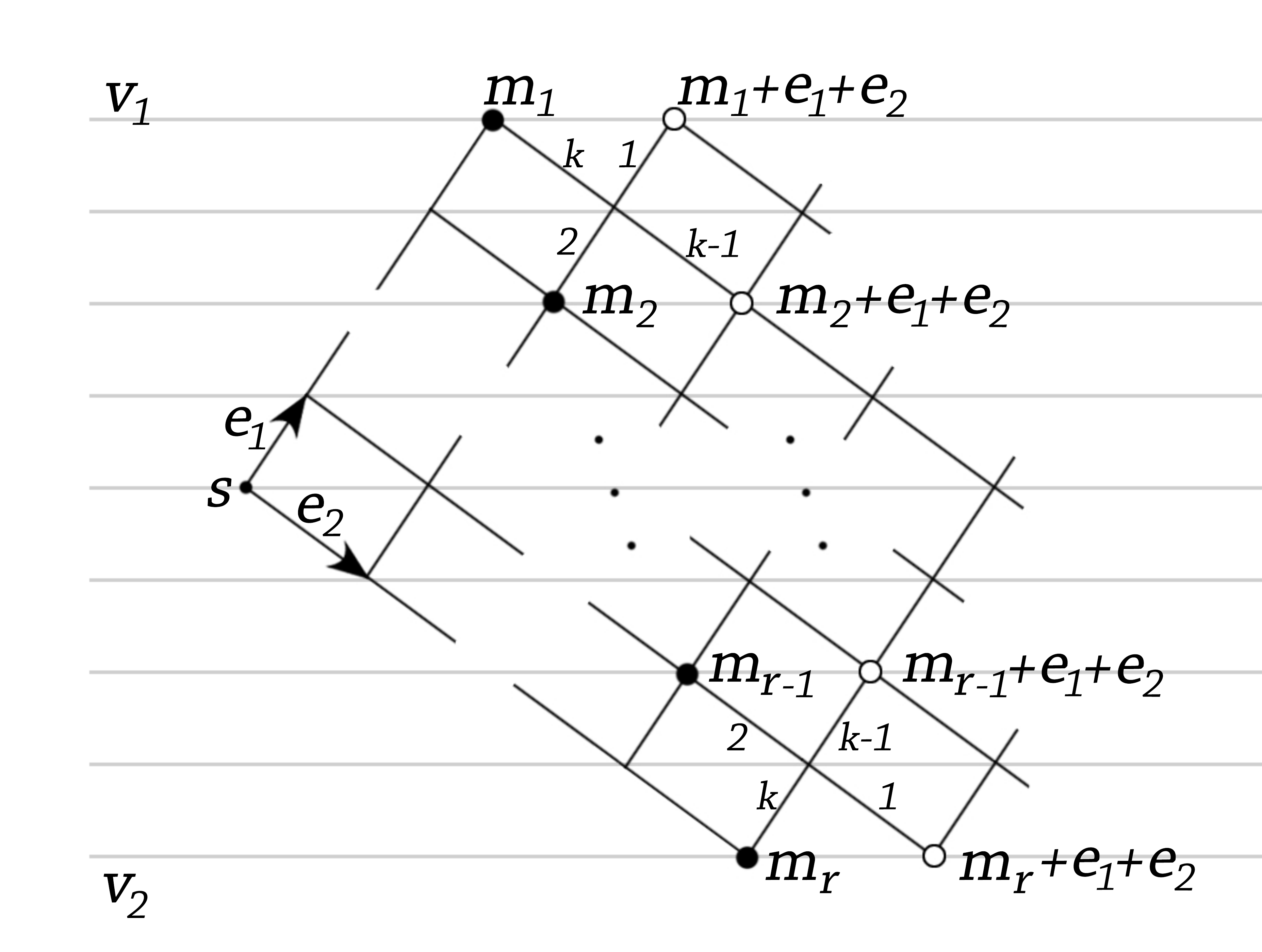}}
\centerline{\figurename{\ 1.} Intersection of $\gamma_s$ and
$\omega_M$}
\end{figure}

Consider a two-dimensional plane containing the lines $v_1$ and
$v_2$, and its intersection with the cone $\omega_M$. We obtain a
stripe-shaped diagram, see Figure 1.

One can also consider a lattice $S$ passing through $s$ with
generating vectors $e_1,e_2$. Since $\<e_i,\rho_j\>=(-1)^{i+j-1}\
(i,j=1,2)$, the derivations $\p_1,\p_2$ send elements of the lattice
$S$ one level up or down respectively. Here levels are of the form
$\{x\in S\ |\ \<x,\rho_2\>=\alpha\},\ \alpha=0,\dots,k$. In particular,
0-level lies on $v_2$ and $k$-level lies on $v_1$. Consider
elements $m_1,\dots,m_r\in S$ as shown on Figure~1. Note that Figure~1
represents the case of even $k$, for odd $k$ arguments are
similar. Let us prove that there exist coefficients $a_1,\dots,a_r\in \overline{\kk(X)^{\TT}}$ such
that
$$
a_1\chi^{m_1}+\dots+a_r\chi^{m_r}\notin \Nil(\p_1+\p_2).
$$
Indeed, we apply $(\p_1+\p_2)^2$ to an element with yet
undetermined coefficients and write the column vector of
coefficients at $\chi^{m_1+e_1+e_2},\dots,
\chi^{m_r+e_1+e_2}$ after decomposition into graded components:

$$
\small
\begin{pmatrix}
a_1 \\ a_2 \\ a_3 \\ \vdots \\ a_{r-1} \\ a_r
\end{pmatrix}
\rightsquigarrow
\begin{pmatrix}
a_1\cdot f_1f_2\cdot k\cdot 1 + a_2\cdot f_1^2\cdot 2\cdot 1 \\
a_2\cdot f_1f_2\cdot ((k-2)\cdot 3 +2\cdot(k-1))+ a_3\cdot f_1^2\cdot 4\cdot 3+a_1\cdot f_2^2\cdot k\cdot (k-1)\\
a_3\cdot f_1f_2\cdot ((k-4)\cdot 5 +4\cdot(k-3))+ a_4\cdot f_1^2\cdot 6\cdot 5+a_2\cdot f_2^2\cdot (k-2)\cdot (k-3)\\
\vdots\\
a_{r-1}\cdot f_1f_2\cdot (2\cdot (k-1)+(k-2)\cdot 3)+ a_r\cdot f_1^2\cdot k\cdot (k-1)+a_{r-2}\cdot f_2^2\cdot 4\cdot 3\\
a_r\cdot f_1f_2\cdot k \cdot 1 + a_{r-1}\cdot f_2^2\cdot 2\cdot 1\\
\end{pmatrix}
$$

We try to find such values $a_1,\dots,a_r$, that these two vectors are
proportional. In other words, $(a_1,\dots,a_r)^T$ has to be an
eigenvector of the following three-diagonal matrix $A$:

$$
\small
\begin{pmatrix}
f_1f_2\cdot k\cdot 1 & * & 0 & 0 & \hdotsfor{1} & 0 \\
* & f_1f_2\cdot ((k-2)\cdot 3 +2\cdot(k-1)) & * & 0 & \hdotsfor{1} & 0 \\
0 & \hdotsfor{1} & \ddots & * & \hdotsfor{1} & \vdots \\
\vdots & \hdotsfor{1} & * & \ddots & \hdotsfor{1} & 0\\
0 & \hdotsfor{1} & 0 & * & f_1f_2\cdot (2\cdot (k-1)+(k-2)\cdot 3) & * \\
0 & \hdotsfor{1} & 0 & 0 & * & f_1f_2\cdot k \cdot 1
\end{pmatrix}
$$

Therefore, it suffices to show that the matrix $A$ possesses a nonzero
eigenvalue $\lambda$, because otherwise applying even powers of
derivation $\p_1+\p_2$ we obtain the following sequence:
$$
\sum_{i=1}^r a_i\chi^{m_i}\longrightarrow \lambda\left(\sum_{i=1}^r
a_i\chi^{m_i+e_1+e_2}\right)\longrightarrow
\lambda^2\left(\sum_{i=1}^r
a_i\chi^{m_i+2(e_1+e_2)}\right)\longrightarrow \dots,
$$
whose members do not vanish and thus $\sum_{i=1}^r
a_i\chi^{m_i}\notin\Nil(\p_1+\p_2)$. We write the characteristic
polynomial of $A$:
$$
\chi(\lambda)=\mathrm{det}(\lambda E -
A)=\lambda^r-\mathrm{tr}A\lambda^{r-1}+\dots+(-1)^r\mathrm{det}A
$$
and compute the coefficient at $\lambda^{r-1}$. One can easily see
that $\mathrm{tr}A$ is a product of two factors, where the first
factor is $f_1f_2$ and the second one is a sum of positive integers. More precisely,

\begin{multline*}
\mathrm{tr}A=f_1f_2\sum_{j=1}^k
j(k+1-j)=f_1f_2\left((k+1)\frac{k(k+1)}{2}-\frac{k(k+1)(2k+1)}{6}\right)=\\
=f_1f_2\frac{k(k+1)(k+2)}{6}\ne0.
\end{multline*}

Hence, $\mathrm{tr}A$ does not vanish and $\chi(\lambda)\ne
\lambda^r$. Therefore we obtain existence of a nonzero $\lambda\in\overline{\kk(X)^{\TT}}$ with $\chi(\lambda)=0$.

Now let $e_1$ and $e_2$ be collinear, i.e. $e_1=te_2,\ t\in \QQ^*$.
Combining the conditions
$$
\<e_1,\rho_2\>=\<te_2,\rho_2\>=-t \Rightarrow t\in\mathbb{Z}_{<0}
$$
and
$$
\<e_2,\rho_1\>=\left\langle\frac{e_1}{t},\rho_1\right\rangle=-\frac1t
\Rightarrow \frac1t\in\mathbb{Z}_{<0},
$$
we see that $t=-1$ and $e_1=-e_2$. The construction from the
previous case allows us to obtain an element of the semigroup
algebra $\overline{\kk(X)^{\TT}}[\omega_M]$ that does not belong to
$\Nil(\p_1+\p_2)$. The proof is completed.
\end{proof}

\section{Concluding remarks}
\label{sec6}

The next corollary follows directly from Lemma $\ref{sameray}$, Proposition
$\ref{commutator}$ and Theorem $\ref{main_theorem}$.

\begin{corollary}
Let $\p_1$ and $\p_2$ be two nonzero homogeneous LNDs on $A$ of
fiber type. Then $\p_1+\p_2$ is a locally nilpotent derivation if
and only if $[\p_1,\p_2]$ is.
\end{corollary}

\begin{example}\label{polynom}
Let us illustrate the obtained results for $A=\kk[x_1,\dots,x_n]$ in the settings of Example~$\ref{ex-1}$. Using for simplicity multiindices ($x^I=x_1^{i_1}\dots x_n^{i_n}$),
we set
$$
\p_1=\lambda_1 x^I \frac{\p}{\p x_i}, \qquad \p_2=\lambda_2 x^J
\frac{\p}{\p x_j},
$$
where $x^I$ and $x^J$ do not contain factors $x_i$ and $x_j$
respectively. Then the following conditions hold:
\begin{itemize}
\item[1)] $[\p_1,\p_2]=0 \Leftrightarrow$ $x^I$ does not depend on $x_j$ and $x^J$ does not depend on $x_i$;
\item[2)] $[\p_1,\p_2]\in \LND(A)\Leftrightarrow \p_1+\p_2\in \LND(A) \Leftrightarrow$  $x^I$ does not depend on $x_j$ or $x^J$ does not depend on $x_i$.
\end{itemize}

Note that in the settings of Example $\ref{polynom}$ Theorem \ref{main_theorem} can be easily obtained from the following proposition.

\begin{proposition}\label{principle}\cite[Principle 5]{GF}
Let $\p\in \LND(A)$ and $f_1,\dots,f_m\in A$ ($m\geqslant 1$). Suppose there exists a permutation $\sigma\in S_m$ such that
$\p f_i\in f_{\sigma(i)}A$ for each $i$. Then in each orbit of
$\sigma$ there is an index $i$ with $\p f_i=0$.
\end{proposition}

Indeed, suppose $\p=\p_1+\p_2\in \LND(A)$. We take $f_1=x_i$,
$f_2=x_j$. Then $\p f_1=\lambda_1 x^I\ne 0$ and
$\p f_2=\lambda_2 x^J\ne 0$. Using Proposition $\ref{principle}$ we
obtain the required condition. The converse is immediate.
\end{example}

\begin{corollary}
Let $\p_1$ and $\p_2$ be two nonzero homogeneous LNDs on $A$ of fiber type. Suppose $e_1+e_2\notin\omega_M$. Then the Lie algebra $L$ generated by $\p_1$ and $\p_2$ over $\kk$ is finite dimensional and consists of locally nilpotent derivations.
\end{corollary}

\begin{proof}
We can assume without loss of generality that $\<e_1,\rho_2\>=0$ and $\<e_2,\rho_1\>=n\geqslant 0$. Then the algebra $L$ is linearly generated by the derivations $\p_1,\p_2$ and $\p_2^{(1)},\dots,\p_2^{(n)}$, where $\p_2^{(i)}=\ad(\p_1)^i\p_2$. So the dimension of $L$ equals to $n+2$. Note also that $\p_2^{(i)},\ i=1,\dots,n$, are locally nilpotent derivations of degree $e_2+ie_1$ respectively, which are Demazure roots corresponding to the ray $\rho_2$. Let us assume that a derivation $\p$ in $L$ is a linear combination of $\p_2,\p_2^{(1)},\dots,\p_2^{(n)}$. Then $\p$ is locally nilpotent because all these LNDs commute. If $\p=\p_1+\lambda_0\p_2+\lambda_1\p_2^{(1)}+\dots+\lambda_n\p_2^{(n)}$, the proof is similar to the proof of sufficiency in Theorem~$\ref{main_theorem}$.
\end{proof}

At the moment the author does not know whether the condition $\deg\,\p_1+\deg\,\p_2\notin\omega_M$ implies that $\p_1+\p_2\in\LND(A)$ for locally nilpotent derivations $\p_1,\ \p_2$ of horizontal type. Let us give a particular result in this direction.

\begin{proposition}
Consider an effectively graded affine algebra
$$
A=\bigoplus_{m\in \omega_M}\widetilde{A_m}.
$$
Let $\p_1$ and $\p_2$ be two locally nilpotent derivations on $A$ of degrees $e_1, e_2\in M$ respectively. If $\{te_1+(1-t)e_2,\ t\in\QQ\}\cap\omega_M=\varnothing$, then $\p_1+\p_2\in \LND(A)$.
\end{proposition}

\begin{proof}
Let us consider the grading of $A$ given by the quotient group $M'=M/(e_1-e_2)$, and denote by $\pi: M\to M'$ the factorization map. Derivations $\p_1$ and $\p_2$ considered as derivations on $M'$-graded algebra are homogeneous of the same degree $\pi(e_1)=\pi(e_2)$. Hence, the sum $\p_1+\p_2$ is $M'$-homogeneous of degree $\pi(e_1)$ as well.

The condition $\{te_1+(1-t)e_2,\ t\in\QQ\}\cap\omega_M=\varnothing$ means that the line passing through the lattice points $e_1,e_2\in M$ does not intersect the cone $\omega_M$. This implies that $\pi(e_1)\notin\pi(\omega_M)$.

Summarizing these facts, we see that $\p$ is a homogeneous derivation on $M'$-graded algebra $A$ and its degree $\pi(e_1)$ does not lie in the weight cone $\pi(\omega_M)$. By Remark $\ref{deg_notin_cone}$ this is sufficient for the local nilpotency of $\p$.
\end{proof}

Unlike the case of derivations of fiber type another implication of Theorem $\ref{main_theorem}$ does not hold for derivations of horizontal type.

\begin{example}
Consider a $\TT$-variety $\AA^2$ with the following action of a
one-dimensional torus $\TT$: $t\cdot (x,y)=(tx,ty)$. The algebra of regular functions $A=\kk[x,y]$ is graded by the weight cone
$\omega_M=\ZZ_{\geqslant 0}$ and $\deg\,x=\deg\,y=1$. Moreover, $\kk(X)^{\TT}=\kk\left(\dfrac xy\right)$.

Consider a homogeneous locally nilpotent derivation
$\p=y\dfrac{\p}{\p x}$ on $A$. Note that $\p\left(\dfrac
xy\right)=1$, hence $\p$ is of horizontal type. Let us show that
two copies of the derivation $\p$ give a counterexample to Theorem
$\ref{main_theorem}$. Condition for commutators given in Proposition
$\ref{commutator}$ for derivations of fiber type also does not hold. Indeed, note that $\deg\,\p=0\in\omega_M$. Taking $\p_1=\p_2=\p$, we
obtain $\deg\,\p_1+\deg\,\p_2=0\in\omega_M$, though
$\p_1+\p_2=2\p\in \LND(A)$ and $[\p_1,\p_2]=[\p,\p]=0\in \LND(A)$.
\end{example}

Finally, we would like to discuss torus actions and homogeneous LNDs coming from actions of reductive algebraic groups, cf. \cite{IA+AL}.

Let $G$ be a reductive algebraic group and $\mathfrak{g}=\Lie G$. We consider a maximal torus $\TT\subset G$ and $\mathfrak{h}=\Lie \TT$. The reductive Lie algebra $\mathfrak{g}$ admits a Cartan decomposition
$$
\mathfrak{g}=\mathfrak{h}\oplus\bigoplus_{\alpha\in\Delta}\mathfrak{g}_{\alpha},
$$
where $\Delta\subset\mathfrak{h}^*$ is the system of roots and $\mathfrak{g}_{\alpha}=\langle e_{\alpha}\rangle$.

Starting with a regular action of $G$ on an affine variety $X$ we obtain the $\TT$-action on $X$ and the corresponding grading of $A=\kk[X]$ by the lattice $M$ of characters of the torus $\TT$. For every $\alpha\in\Delta$, the nilpotent element $e_{\alpha}\in\mathfrak{g}$ defines a homogeneous locally nilpotent derivation $\p_{\alpha}$ on $\kk[X]$ of degree $\alpha$.

\begin{lemma}\label{alpha-beta}
Let $\alpha,\beta\in\Delta$. Then $e_{\alpha}+e_{\beta}$ is nilpotent if and only if $\alpha+\beta\ne 0$.
\end{lemma}

\begin{proof}
Assume $\beta=-\alpha$. In this case $e_{\alpha}$, $e_{-\alpha}$ and $[e_{\alpha},e_{-\alpha}]$ form an $\mathfrak{sl}_2$-triple and, thus, $e_{\alpha}+e_{-\alpha}$ is semisimple.

For $\alpha\ne -\beta$ we can consider a hyperplane in $\mathfrak{h}^*$ such that $\alpha$ and $\beta$ are contained in the same open half-space. This hyperplane determines a one-dimensional torus $T$, and the roots $\alpha$ and $\beta$ are positive with respect to $T$. Looking at the weights occuring in the decomposition of the element $\ad(e_{\alpha}+e_{\beta})^N(x)$ for some homogeneous $x\in\mathfrak{g}$ we obtain that $e_{\alpha}+e_{\beta}$ is nilpotent.
\end{proof}

Using this lemma and the result of Theorem $\ref{main_theorem}$ we obtain the following

\begin{corollary}\label{alpha-beta-cone}
Assume that $\alpha+\beta\ne 0$. If $\p_{\alpha}$ and $\p_{\beta}$ are of fiber type and $\p_{\alpha}+\p_{\beta}\in\LND(A)$, then $\alpha+\beta$ does not belong to the weight cone $\omega_M$ corresponding to the action of the maximal torus $\TT$ in $G$ on $A$.
\end{corollary}

\begin{example}
Consider a natural $GL_n(\kk)$-action on the affine space $\AA^n$. The affine space $\AA^n$ gets a structure of a toric variety under the action of the maximal torus
$$
\TT^n=\{\diag(t_1,\dots,t_n)\}\subseteq GL_n(\kk).
$$
The weights of $x_1,\dots,x_n$ are $\varepsilon_1,\dots,\varepsilon_n$ respectively. A root $\alpha$ determines an LND $\p_{\alpha}$:
$$
\alpha=\varepsilon_i-\varepsilon_j, i\ne j \qquad \rightsquigarrow \qquad \p_{\alpha}=x_i\frac{\p}{\p x_j}.
$$
It is easy to see that
$$
\p_{\alpha}+\p_{\beta}=x_i\frac{\p}{\p x_j}+x_k\frac{\p}{\p x_l}\in\LND(A) \quad \Leftrightarrow \quad (i,j)\ne(l,k)\quad \Leftrightarrow \quad\alpha\ne -\beta.
$$
Moreover, $\alpha+\beta=\varepsilon_i-\varepsilon_j+\varepsilon_k-\varepsilon_l\notin \omega_M=\ZZ_{\geqslant 0}^n$ for $(i,j)\ne(l,k)$. This illustrates Corollary~$\ref{alpha-beta-cone}$.
\end{example}

\begin{example}
Consider a natural $SL_n(\kk)$-action on the affine space $\AA^n$. The maximal torus
$$
\TT^{n-1}=\{\diag(t_1,\dots,t_n)\ |\ \prod_i\,t_i=1\}\subseteq SL_n(\kk)
$$
acts on $\AA^n$ with complexity one. The weights of $x_1,\dots,x_n$ are $\varepsilon_1,\dots,\varepsilon_{n-1}$ and $-\varepsilon_1-\dots-\varepsilon_{n-1}$ respectively, thus, $\omega_M=M$. Therefore, there exist no locally nilpotent derivations of fiber type. Taking $\p_{\alpha}+\p_{\beta}$ with $\alpha+\beta\ne 0$ we obtain one more counterexample to Theorem $\ref{main_theorem}$ in case of derivations of horizontal type.
\end{example}

\section*{Acknowledgement}

The author is grateful to her supervisor Ivan~Arzhantsev for posing the problem and permanent support in preparing the text of this paper, and to Polina~Kotenkova for useful suggestions.


\end{document}